\documentclass[11pt, reqno]{amsart}
\pdfoutput=1

\usepackage[utf8]{inputenc}
\usepackage[T1]{fontenc}
\usepackage{lmodern}
\usepackage[draft=false,kerning=true]{microtype}

\usepackage{a4wide}
\usepackage{amsthm}
\usepackage{amssymb}
\usepackage{amsmath}
\usepackage{colonequals}
\usepackage{mathtools}
\usepackage{tikz}

\usepackage{hyperref}

\DeclarePairedDelimiter{\abs}{\lvert}{\rvert}
\DeclarePairedDelimiter{\card}{\lvert}{\rvert}
\DeclarePairedDelimiter{\floor}{\lfloor}{\rfloor}
\DeclarePairedDelimiter{\fractional}{\{}{\}}
\DeclarePairedDelimiterXPP{\oh}[1]{o}(){}{#1}
\DeclarePairedDelimiterXPP{\Oh}[1]{\mathcal{O}}(){}{#1}

\DeclarePairedDelimiter{\set}{\{}{\}}
\DeclarePairedDelimiterX{\setm}[2]{\{}{\}}{#1\,\delimsize\vert\,\mathopen{}#2}

\newtheorem{thm}{Theorem}

\newtheorem{lem}[thm]{Lemma}
\newtheorem{cor}[thm]{Corollary}
\theoremstyle{remark}

\makeatletter
\def\namedlabel#1#2{\begingroup
    #2%
    \def\@currentlabel{#2}%
    \phantomsection\label{#1}\endgroup
}
\makeatother

\newcommand{\TODO}[1]%
{\par\fbox{\begin{minipage}{0.9\linewidth}\textbf{TODO:}
      #1\end{minipage}}\par}

\begin{document}

\thispagestyle{plain}

\title{On the minimal Hamming weight of a multi-base representation}

\author{Daniel Krenn}
\address{Daniel Krenn \\
Department of Mathematics \\
Alpen-Adria-Universit\"at Klagenfurt
Universit\"atsstra\ss e 65--67 \\
9020 Klagenfurt \\
Austria
}
\email{\href{mailto:math@danielkrenn.at}{math@danielkrenn.at} \textit{or}
  \href{mailto:daniel.krenn@aau.at}{daniel.krenn@aau.at}}

\thanks{Daniel Krenn is supported 
    by the Austrian Science Fund (FWF): P\,28466-N35.}

\author{Vorapong Suppakitpaisarn}
\address{Vorapong Suppakitpaisarn \\
  Graduate School of Information Science and Technology \\
  The University of Tokyo \\
  7-3-1 Hongo, Bunkyo-ku, Tokyo 113-0033 \\
  Japan
}
\email{\href{mailto:vorapong@is.s.u-tokyo.ac.jp}{vorapong@is.s.u-tokyo.ac.jp}}

\author{Stephan Wagner}
\address{Stephan Wagner \\
Department of Mathematical Sciences \\
Stellenbosch University \\
Private Bag X1 \\
Matieland 7602 \\
South Africa
}
\email{\href{mailto:swagner@sun.ac.za}{swagner@sun.ac.za}}

\thanks{Stephan Wagner is supported by the National Research Foundation of South Africa, grant 96236.}


\date{\today}

\begin{abstract}
  Given a finite set of bases $b_1$, $b_2$, \dots, $b_r$ (integers
  greater than $1$),
  a multi-base representation of an integer~$n$ is
  a sum with summands $db_1^{\alpha_1}b_2^{\alpha_2} \cdots b_r^{\alpha_r}$,
  where the $\alpha_j$ are nonnegative integers and
  the digits $d$ are taken from a fixed finite set. We consider multi-base representations
  with at least two bases that are multiplicatively independent. Our main result states
  that the order of magnitude of the minimal Hamming weight of an integer~$n$, i.e., the
  minimal number of nonzero summands in a representation of~$n$,
  is $\log n / (\log \log n)$. This is independent of the number of bases, the
  bases themselves, and the digit set.

  For the proof, the existing upper bound for prime bases is
  generalized to multiplicatively independent bases; for the required
  analysis of the natural greedy algorithm, an auxiliary result in
  Diophantine approximation is derived. The lower bound follows by a
  counting argument and alternatively by using communication
  complexity; thereby improving the existing bounds and closing the
  gap in the order of magnitude. This implies also that the
  greedy algorithm terminates after $\mathcal{O}(\log n/\log \log n)$
  steps, and that this bound is sharp.
\end{abstract}

\maketitle

\section{Introduction}
\label{sec:intro}

\subsection{Multi-base representations}

Let a finite set $\set{b_1,b_2,\ldots,b_r}$ of
\emph{bases} (integers greater than $1$) be given, along with a finite
set $D$ of nonnegative integers that includes $0$. The elements of $D$
will be called \emph{digits}. We let
\begin{equation*}
  \mathcal{B} = \setm[\big]{b_1^{\alpha_1} b_2^{\alpha_2} \cdots b_r^{\alpha_r}}{
    \text{$\alpha_1$, $\alpha_2$, \dots, $\alpha_r$ nonnegative integers}}
\end{equation*}
be the free monoid generated by $b_1$, $b_2$, \dots, $b_r$;
the elements of $\mathcal{B}$ are called \emph{power-products}.
A \emph{multi-base representation} of a positive integer~$n$ is a
representation of the form
\begin{equation}\label{eq:multi-base}
  n = \sum_{B \in \mathcal{B}} d_B B,
  \tag{$\divideontimes$}
\end{equation}
where $d_B \in D$ for all $B \in \mathcal{B}$.

For simplicity, we make the natural
assumption that every positive integer has at least one such
representation, which implies in particular that $1 \in D$. 
We will also assume that the bases $b_1$, $b_2$, \dots, $b_r$ are
multiplicatively independent, i.e., the only integers
$\alpha_1$, $\alpha_2$, \dots, $\alpha_r$ for which
\begin{equation*}
  b_1^{\alpha_1}b_2^{\alpha_2} \cdots b_r^{\alpha_r} = 1
\end{equation*}
are $\alpha_1 = \alpha_2 = \cdots = \alpha_r = 0$. Intuitively, this means
that there is no ``redundancy'' in the set of bases.

Note that we obtain the
standard base-$b$ representation for $r = 1$, base~$b_1 = b$ and
digit set~$D = \set{0,1,\ldots,b-1}$.

\subsection{Notes on the set-up}
\label{sec:setup}

The set-up for multi-base representations that we described is quite standard
(except possibly for the multiplicative independence). However, our
proofs still apply with the following modifications:
\begin{itemize}
\item All digits~$d_B$ in the multi-base
  representations~\eqref{eq:multi-base} of~$n$ are assumed to be in
  $\Oh{\log n}$ (in contrast to a finite, nonnegative digit set).
\item All exponents~$\alpha_j$ in the multi-base
  representations~\eqref{eq:multi-base} of~$n$ are assumed to be in
  $\Oh{\log n}$. This is essentially trivial if all digits are nonnegative (see
  Section~\ref{sec:lower-existing}), but we can also allow negative digits
if this additional assumption is imposed.
\item At least two of the bases are assumed to be multiplicatively
  independent (in contrast to the entire set being multiplicatively independent).
\end{itemize}

\subsection{Hamming weight}

Of course, only finitely many terms of the
sum~\eqref{eq:multi-base} can be nonzero.
The number of these terms is called the
\emph{Hamming weight} of a representation. The Hamming weight is a
measure of how efficient a certain representation is.
A multi-base representation of an integer~$n$ is called \emph{minimal}
if it minimizes the Hamming weight among all multi-base
representations of~$n$ with the same bases and digit set.

An overview on previous works concerning the Hamming weight of
multi-base representations will follow in
Sections~\ref{sec:greedy} to~\ref{sec:single-base}.
At this point, we only mention that
the Hamming weight of single-base representations
has been thoroughly studied (see Section~\ref{sec:single-base}),
not only in the case of the standard set $\set{0,1,\ldots,b-1}$
of digits, but also for more general types of digit sets. Both the
worst case (maximum) and the average order of magnitude of the Hamming
weight are $\log n$.

\subsection{Main result}
In this short note, we investigate the Hamming
weight of multi-base representations and find that the Hamming weight
can be reduced---even in the worst case---by using multi-base
representations. However, the reduction compared to single-base representations
is fairly small. Perhaps surprisingly, the order of magnitude is independent of the
number~$r$ of bases (provided only that $r \geq 2$), the set of bases and the
set of digits: it is always $\frac{\log n}{\log \log n}$.

The precise statement is as follows.

\begin{thm}\label{thm:main}
  Suppose that $r \geq 2$, and that the multiplicatively independent bases
  $b_1$, $b_2$, \dots, $b_r$ and the digit set~$D$ are
  such that every positive integer $n$ has a representation of the
  form~\eqref{eq:multi-base}. There exist two positive constants $K_1$
  and $K_2$ (depending on $b_1$, $b_2$, \dots, $b_r$ and $D$) such that
  the following hold:
  \begin{itemize}
  \item[\namedlabel{itm:upper}{(U)}]
    For all integers $n > 2$, there exists a representation of the
    form~\eqref{eq:multi-base} with Hamming weight at most
    $K_1 \frac{\log n}{\log \log n}$.
  \item[\namedlabel{itm:lower}{(L)}]
    For infinitely many positive integers $n$, there is no
    representation of the form~\eqref{eq:multi-base} whose Hamming
    weight is less than $K_2 \frac{\log n}{\log \log n}$.
  \end{itemize}
\end{thm}

The upper bound of this theorem needs weaker assumptions on the bases
than the result of Dimitrov, Jullien and
Miller~\cite{Dimitrov-Jullien-Miller:1998:algor-for-modul-expon}: They
require that all the bases~$b_1$, \dots, $b_r$ are primes,\footnote{
  The proof of the bound
  in~\cite{Dimitrov-Jullien-Miller:1998:algor-for-modul-expon} is
  carried out for double-base representations with bases~$2$ and $3$,
  and it is stated that it generalizes to sets of bases being finite
  sets of primes.}
whereas we
only need that (two of) the bases are multiplicatively independent.
The order of magnitude of both bounds coincides. We will prove the
bound~\ref{itm:upper} for our general multi-base set-up
in Section~\ref{sec:upper} by analyzing the Greedy algorithm.

The best known lower bound\footnote{ When we speak of a ``lower
  bound'', say $L(n)$, in this paper, we mean that there exist
  infinitely many positive integers~$n$ which do not have a
  representation with Hamming weight less than~$L(n)$.}
for the minimal Hamming weight seems to be
of order $\frac{\log n}{\log \log n \,\cdot\, \log \log \log n}$ (see
Dimitrov and Howe~\cite{Dimitrov-Howe:2011:length-double-base-repr})
for double-base representations with bases~$2$ and $3$. Yu, Wang, Li
and Tian~\cite{Yu-Wang-Li-Tian:2013:length-triple-base} extend this
result to triple-base representations with bases~$2$, $3$ and $5$.
Our lower bound~\ref{itm:lower} closes the gap to the upper
bound in the order by getting rid of the factor~$\log \log \log n$
in the denominator. We show this result in
Section~\ref{sec:lower} by a counting argument and in
Section~\ref{sec:communication-complexity} by using communication
complexity.

\subsection{Background on multi-base representations}

Motivation for studying multi-base representations comes from
fast and efficient arithmetical operations. One particular starting
point is~\cite{Dimitrov-Jullien-Miller:1998:algor-for-modul-expon},
where double-base and multi-base representations are used for modular
exponentiation.  Beside many other references,
\cite{Avanzi-Dimitrov-Doche-Sica:2006:double-base,
  Dimitrov-Imbert-Mishra:2008:double-base,
  Dimitrov-Jullien-Miller:1999:double-base} describe the usage of
double-base systems for cryptographic applications; the typical bases
used are~$2$ and~$3$.

Questions such as: does every integer have a multi-base representation,
or: what is the smallest number that cannot be represented in a certain
system, are also of great interest; cf.\@
\cite{Berthe-Imbert:2009,
  Bertok:2013:diff-power-products,
  Bertok-Hajdu-Luca-Sharma:2017:number-non-zero-multi-base,
  Krenn-Thuswaldner-Ziegler:2013:belcher}.
The number of multi-base representations has also been analyzed; see
\cite{Krenn-Ralaivaosaona-Wagner:2014:multi-base-asy,
  Krenn-Ralaivaosaona-Wagner:ta:multi-base-asy-full}.

\subsection{Greedy algorithm}
\label{sec:greedy}

Let us come back to multi-base representations in this work's set-up.
The natural greedy algorithm finds a multi-base representation of a
nonnegative integer~$n$ successively by
\begin{itemize}
\item adding the largest power-product~$B \in \mathcal{B}$ less than or equal to~$n$ 
to the representation, and
\item continuing in the same manner with $n-B$.
\end{itemize}
The greedy algorithm does not produce a minimal representation in
general. For instance, for double-base representations with bases~$2$
and $3$, the smallest counter-example is
\begin{equation*}
  41 = 2^23^2 + 2^2 + 1 = 2^5 + 3^2.
\end{equation*}

The upper bound for the minimal Hamming weight is derived by Dimitrov,
Jullien and
Miller~\cite{Dimitrov-Jullien-Miller:1998:algor-for-modul-expon} by
analyzing the greedy algorithm (as mentioned for prime bases). This is
also our approach in this paper. Our result translates to the following corollary,
which is a direct consequence of the proof and the statement of Theorem~\ref{thm:main}.

\begin{cor}\label{cor:greedy}
  Suppose that $r \geq 2$, and that the multiplicatively independent
  bases $b_1$, $b_2$, \dots, $b_r$ and the digit set~$D$ are such that
  every positive integer $n$ has a representation of the
  form~\eqref{eq:multi-base}. Then, the natural greedy algorithm with
  input~$n$ terminates after $\Oh[\big]{\frac{\log n}{\log \log n}}$ steps,
  and this bound is sharp. The output is a representation containing
  only digits~$0$ and $1$.
\end{cor}

Note that this corollary is valid if the greedy algorithm is suitably
preprocessed. To make this more precise, the algorithm needs
representations with only digits~$0$ and $1$ for all numbers
from $0$ to some $N_0$. This $N_0$ is to be found in the proof of
Theorem~\ref{thm:main}, part~\ref{itm:upper}; it might
actually be huge (if it can even be calculated with reasonable
effort). On the other hand, relaxing the condition on the digits
being only~$0$ and $1$ for the numbers up to~$N_0$ also suffices for
the validity of Corollary~\ref{cor:greedy}.

Yu, Wang, Li and Tian~\cite{Yu-Wang-Li-Tian:2013:length-triple-base}
use the proof of the $\Oh[\big]{\frac{\log n}{\log \log n}}$ bound
of~\cite{Dimitrov-Jullien-Miller:1998:algor-for-modul-expon} for
double-base representations with bases~$2$ and $3$ to show the same
bound for triple-base representations with bases~$2$, $3$ and $5$.

It is already mentioned
in~\cite{Dimitrov-Jullien-Miller:1998:algor-for-modul-expon} that
their upper bound of the Hamming weight of the representations
obtained by the greedy algorithm is best possible. Such a lower bound
is also derived
in~\cite{Chalermsook-Imai-Suppakitpaisarn:2015:lower-bound-double-base}.

\subsection{Lower bounds}
\label{sec:lower-existing}

Clearly, the minimal Hamming weight of integers~$n \in \mathcal{B}$
is~$1$. So a goal related to lower bounds is to find sequences of
integers with large minimal Hamming weight.

As mentioned, Dimitrov and
Howe~\cite{Dimitrov-Howe:2011:length-double-base-repr} and Yu, Wang,
Li and Tian~\cite{Yu-Wang-Li-Tian:2013:length-triple-base} state the
existence of a constant~$K_2$ and the existence of infinitely many
integers~$n$ whose minimal Hamming weight is greater than $K_2
\frac{\log n}{\log \log n \,\cdot\, \log \log \log n}$ for
representations with bases~$2$ and $3$, and bases~$2$, $3$ and $5$,
respectively.

\subsection{Distribution of the Hamming weight}
\label{sec:distribution}

Beside the minimal Hamming weight of an integer~$n$, the expected
Hamming weight of a random multi-base representation of~$n$ and
more generally the distribution of the
Hamming weight of all representations of~$n$ have been studied. In
\cite{Krenn-Ralaivaosaona-Wagner:2014:multi-base-asy,
  Krenn-Ralaivaosaona-Wagner:ta:multi-base-asy-full}, an asymptotic
formula of the form $K (\log n)^r + \Oh[\big]{(\log n)^{r-1} \log\log n}$
for the expected Hamming weight of a random
representation of an integer~$n$ is derived with explicit constant~$K$;
see~\cite[Theorem~IV]{Krenn-Ralaivaosaona-Wagner:ta:multi-base-asy-full}.
The order of magnitude $(\log n)^r$ of this result depends,
in contrast to the minimal Hamming weight, on the number~$r$ of bases.
Moreover, it
is shown in \cite{Krenn-Ralaivaosaona-Wagner:2014:multi-base-asy,
  Krenn-Ralaivaosaona-Wagner:ta:multi-base-asy-full}
that the Hamming weight asymptotically follows a Gaussian
distribution, and an asymptotic expression for the variance is provided as well.

\subsection{Single-base representations}
\label{sec:single-base}

For completeness, we also provide some background on (redundant) single-base
representations, i.e., representations with $r = 1$ and an integer
base~$b_1=b$, but a digit set that might differ from the standard choice
$\{0,1,\ldots,b-1\}$.

Papers \cite{Heuberger-Muir:2009:closest-one} and
\cite{Phillips-Burgess:2004:minim-weigh} provide a way to compute
minimal representations. The minimal Hamming weight of different kinds of single-base
representations is studied in \cite{Cohen:2005:analy-flexible-window,
  Morain-Olivos:1990, muirstinson:minimality,
  Solinas:2000:effic-koblit, Thuswaldner:1999}. One particular
representation, which often is minimal, is the so-called non-adjacent
form (cf.\@ \cite{Reitwiesner:1960,
  Heuberger-Krenn:2013:wnafs-optimality}); it uses a signed digit set,
i.e., a digit set containing also negative integers. Grabner and
Heuberger~\cite{Grabner-Heuberger:2006:Number-Optimal} count
representations with minimal Hamming weight for such a signed digit set.

\section{The upper bound}
\label{sec:upper}

The proof of the first statement of Theorem~\ref{thm:main} follows from
an analysis of the natural greedy algorithm and is based on some results
from Diophantine approximation.

The following lemma is the statement corresponding to the result of
Tijdeman~\cite{Tijedeman:1974:distance-small-primes} on which the
analysis of Dimitrov, Jullien and
Miller in~\cite{Dimitrov-Jullien-Miller:1998:algor-for-modul-expon} is
based on.

\begin{lem}\label{lem:Dioph-lemma}
  There are positive constants $C$ and $\kappa$ with the following
  property: for every integer $n>1$, there
  is an element $B \in \mathcal{B}$ such that
  \begin{equation*}
    n e^{-C (\log n)^{-\kappa}} \leq B \leq n.
  \end{equation*}
\end{lem}

\begin{proof}

It clearly suffices to prove the statement in the case where $r=2$;
let us use the abbreviations $p = b_1$, $q = b_2$, and set $\lambda =
\log_p q$. Since $p$ and $q$ are multiplicatively independent, $\lambda$
is irrational, which will be crucial for us. 

Let $\fractional{x} = x - \floor{x}$ denote the fractional part of a
real number $x$. As a first step, we consider the sequence
$\Lambda_M = (\fractional{\lambda m})_{m = 0}^{M-1}$
and show that its ``gaps'' (intervals that do not contain a value of
$\Lambda_M$) can be bounded in terms of $M$. The structure of these
gaps is in fact very well understood (see
\cite{Alessandri-Berthe:1998}), but we only require an upper bound.

Recall that the \emph{discrepancy} of $\Lambda_M$ is given by 
\begin{equation*}
  D(\Lambda_M) = \sup_J \abs*{\frac{1}{M} \card{J \cap \Lambda_M} - \mu(J)},
\end{equation*}
where $\mu$ denotes the Lebesgue measure and the supremum is taken
over all intervals $J \subseteq [0,1]$. The discrepancy is obviously
an upper bound on the length of the largest gap in $\Lambda_M$ (i.e.,
the Lebesgue measure of the largest interval
$J$ such that $\card{J \cap \Lambda_M} = 0$). Sequences of the form
$(\fractional{\lambda m})_{m \geq 0}$ and their discrepancy have been
investigated quite thoroughly: let $\gamma$ be the \emph{irrationality
  measure} of $\lambda$, which is defined as the infimum of all
exponents $\nu$ for which there are at most finitely many integer
solutions $(a,b)$ to the inequality
\begin{equation*}
  \abs*{\lambda - \frac{a}{b}} < \frac{1}{b^{\nu}}.
\end{equation*}
Then one has $D(\Lambda_M) = \Oh[\big]{M^{-1/(\gamma-1) + \epsilon}}$ for
every $\epsilon > 0$;
see~\cite[Chapter 2.3, Theorem 3.2]{Kuipers-Niederreiter:1974}.
The fact that the irrationality measure $\gamma$ is finite in
our case, where $\lambda = \log_p q$, is a simple consequence of
Baker's theory of linear forms in logarithms; see~\cite{Baker:1990}
for a general reference. Bugeaud~\cite{Bugeaud:2015:effective} even
provides explicit bounds for this specific case.

Fix a positive constant $\kappa < 1/(\gamma-1)$ and a positive
constant $C_1$ such that 
$$D(\Lambda_M) \leq C_1M^{-\kappa}$$
for all $M \geq 1$. We set
$M = \lceil \log_q n \rceil$ and consider the interval from
$\fractional{\log_p n} - C_1 M^{-\kappa}$ to $\fractional{\log_p n}$.
Since the discrepancy is an upper bound on all gaps in $\Lambda_M$, we know that there must be an $m \in \set{0,1,\ldots,M-1}$ such that
\begin{equation*}
  \fractional{\log_p n} - C_1 M^{-\kappa} \leq \fractional{\lambda m}
  \leq \fractional{\log_p n}.
\end{equation*}
Note that if $\fractional{\log_p n} \leq C_1 M^{-\kappa}$, we may simply choose $m=0$.

Since $\lambda m \leq \lambda (M-1) \leq \log_p q \log_q n = \log_p n$,
it follows that there is a nonnegative integer $\ell$ such that
\begin{equation*}
  \log_p n - C_1 M^{-\kappa} \leq \ell + \lambda m \leq \log_p n,
\end{equation*}
which is equivalent to
\begin{equation*}
  \log n - (C_1 \log p) M^{-\kappa} \leq \ell \log p + m \log q \leq \log n.
\end{equation*}
This in turn implies that there exist nonnegative $\ell$ and $m$
such that
\begin{equation*}
  n e^{-C (\log n)^{-\kappa}} \leq p^{\ell} q^m \leq n,
\end{equation*}
where $C = (C_1 \log p) (\log q)^{\kappa}$. This proves the lemma.
\end{proof}

Now we are ready to prove statement~\ref{itm:upper} of Theorem~\ref{thm:main}.

\begin{proof}[Proof of Theorem~\ref{thm:main}, part~\ref{itm:upper}]
Take $C$
and $\kappa$ as in the lemma, and note that
\begin{equation*}
  \frac{\log(Cn/(\log n)^{\kappa})}{\log \log(Cn/(\log n)^{\kappa})}
  = \frac{\log n}{\log \log n} - \kappa + O \Big( \frac{1}{\log \log n} \Big).
\end{equation*}
Let $N_0$ be large enough so that $C/(\log n)^{\kappa} < \frac12$ as well as
\begin{equation}\label{eq:choice-N0:kappa2}
  \frac{\log(Cn/(\log n)^{\kappa})}{\log \log(Cn/(\log n)^{\kappa})}
  \leq \frac{\log n}{\log \log n} - \frac{\kappa}{2}
\end{equation}
for all $n > N_0$. Moreover, choose a constant $K_1 \geq
\frac{2}{\kappa}$ sufficiently large so that every positive integer $n
\in \set{3,4,\ldots,N_0}$ has a representation of the
form~\eqref{eq:multi-base} of Hamming weight at most
$\min \set[\big]{\frac{K_1 \log n}{\log \log n},
  \frac{K_1 \log N_0}{\log \log N_0}}$.

Now it follows by induction that in fact every integer $n > 2$ has a
representation whose Hamming weight is at most
$\frac{K_1 \log n}{\log \log n}$. For $n \leq N_0$,
this holds by our choice of $N_0$ and
$K_1$. For $n > N_0$, Lemma~\ref{lem:Dioph-lemma} guarantees the
existence of an element $B \in \mathcal{B}$ for which
\begin{equation}\label{eq:existence-B:bound}
  0 \leq n - B \leq n - n e^{-C (\log n)^{-\kappa}}
  \leq \frac{Cn}{(\log n)^{\kappa}}.
\end{equation}
The latter inequality follows by taking advantage of
the classic bound~$1-e^{-x} \leq x$.
The number $n-B$ therefore has a representation whose Hamming weight
is at most
\begin{equation*}
  K_1 \cdot \frac{\log(Cn/(\log n)^{\kappa})}{\log \log(Cn/(\log n)^{\kappa})}
  \leq K_1 \frac{\log n}{\log \log n} - \frac{K_1\kappa}{2}
  \leq K_1 \frac{\log n}{\log \log n} - 1
\end{equation*}
because of \eqref{eq:choice-N0:kappa2}.
The bound~\eqref{eq:existence-B:bound} and our assumption
$C/(\log n)^{\kappa} < \frac12$ imply
$n-B \leq Cn/(\log n)^{\kappa} < \frac{n}{2}$, so we must have $n -
B < B$, thus the element $B$ does not occur in the representation of
$n-B$ (i.e., its coefficient $d_B$ is zero). So we can add $B$ to the
representation of $n-B$ to obtain a multi-base representation of the
form~\eqref{eq:multi-base} whose Hamming weight is at most
$\frac{K_1 \log n}{\log \log n}$. This completes the induction and thus the
proof of the desired upper bound.
\end{proof}

\section{The lower bound}
\label{sec:lower}

The second statement~\ref{itm:lower} of Theorem~\ref{thm:main} is proven by means of a
simple counting argument.
We will use the assumptions made in
Section~\ref{sec:setup}. Multiplicative independence is not actually required, though.

Note first that in any representation of the
form
\begin{equation*}
  n = \sum_{B \in \mathcal{B}} d_B B,
\end{equation*}
with nonnegative $d_B \in D$, a digit
$d_B$ can only be nonzero if $B \leq n$. The number $B$, on the other hand, can be represented as
\begin{equation*}
  B = b_1^{\alpha_1} b_2^{\alpha_2} \cdots b_r^{\alpha_r}
\end{equation*}
for some nonnegative integers $\alpha_1$, $\alpha_2$, \dots, $\alpha_r$ by
definition. We must have
\begin{equation*}
  0 \leq \alpha_j \leq \log_{b_j} B,
\end{equation*}
giving us $1 + \floor{\log_{b_j} B}$ possible values for $\alpha_j$.
This justifies our assumption (Section~\ref{sec:setup})
that the number of possible values of
$\alpha_j$ is bounded by $c_j \log n$ for some constant~$c_j$.

\begin{proof}[Proof of Theorem~\ref{thm:main}, part~\ref{itm:lower}]
For the moment, let $N$ be an arbitrary positive integer; later, we will choose
$N=2^s$ and let $s\to\infty$. Let $\mathcal{B}_N \subseteq \mathcal{B}$
be the set of power-products appearing in some multi-base representation of some
integer in the set $\set{1,2,\dots,N}$. We head for a bound for
$\abs{\mathcal{B}_N}$. As mentioned, we have $d_B = 0$ for all
$B > N$, so all such integers $B$ do not contribute to multi-base representations
of numbers in $\{1,2,\ldots,N\}$ and are therefore not contained in $\mathcal{B}_N$.

By the considerations above, we have
\begin{equation*}
  \abs{\mathcal{B}_N} \leq
  T(N) \coloneqq \prod_{j=1}^r (c_j \log N)
  = (\log N)^r \prod_{j=1}^r c_j
\end{equation*}
as $N\to\infty$. The number $R_K(N)$ of
representations using only the power-products in $\mathcal{B}_N$ and having
Hamming weight at most $K$ is bounded above by
\begin{equation*}
  R_K(N) \leq \sum_{k=1}^K \binom{T(N)}{k} (\card{D}-1)^k,
\end{equation*}
since we have at most $\binom{T(N)}{k}$ choices for those
$B\in\mathcal{B}_N$ with nonzero digits $d_B$, and at most
$(\card{D}-1)^k$ choices for the digits. A crude estimate gives us, at least
for $K \leq T(N)/2$,
\begin{align*}
  R_K(N) &\leq \binom{T(N)}{K} \sum_{k=1}^K  (\card{D}-1)^k \\
&\leq \binom{T(N)}{K} \card{D}^K \\
&\leq \bigl(\card{D}\, T(N)\bigr)^K.
\end{align*}
We claim that for every positive constant~$K_2 < \frac{1}{r}$, the following
holds: for all sufficiently large positive integers~$s$, there is
an integer $n \in \set{2^{s-1}+1,2^{s-1}+2,\ldots,2^{s}}$ without a representation
whose Hamming weight is less than $K_2 \frac{\log n}{\log \log n}$.
This implies that there are infinitely many values of positive integers $n$
for which there is
no representation whose Hamming weight is less than or equal to
$K_2 \frac{\log n}{\log \log n}$, completing the proof.

To prove the claim, suppose all integers in the set
$\set{2^{s-1}+1,2^{s-1}+2,\ldots,2^{s}}$ have a representation whose
Hamming weight is at most $K$. Then we must have
\begin{equation*}
  \bigl( \card{D}\, T(2^{s}) \bigr)^K \geq R_K(2^s) \geq 2^{s-1}.
\end{equation*}
Taking logarithms yields
\begin{equation*}
  K \geq \frac{(s-1) \log 2}{\log T(2^{s}) + \log \card{D}}
  = \frac{(s-1) \log 2}{r \log s + \Oh{1}}
  > K_2 \frac{\log(2^s)}{\log \log(2^s)}
\end{equation*}
for sufficiently large $s$. The claim follows.
\end{proof}

\section{From the point of view of communication complexity}
\label{sec:communication-complexity}

In this section, we will provide an alternative proof based on
communication complexity, to show that the upper bound obtained in
Section~\ref{sec:upper} is asymptotically tight; i.e., we
prove~\ref{itm:lower} of Theorem~\ref{thm:main}.

As mentioned, we use communication complexity to prove the
statement. Consider the situation where Alice and Bob both hold $\ell$
bits of information (or equivalently a nonnegative integer less than
$2^{\ell}$), denoted by the messages~$m_{\mathrm{Alice}}$ and~$m_{\mathrm{Bob}}$.
Bob wants to check if they hold the same information. To do that, Alice
can send some message (according to some protocol) to Bob.
Every time Bob got a bit of information
from Alice, he can announce ``equal'' if he is sure that
$m_{\mathrm{Alice}} = m_{\mathrm{Bob}}$, ``not equal'' when he is sure
that $m_{\mathrm{Alice}} \neq m_{\mathrm{Bob}}$, or ``more information''
to request more information on~$m_{\mathrm{Alice}}$ from Alice. Alice
wants to minimize the number of bits that she sends to Bob; see
Yao~\cite{Yao:1979:complexity-questions}.

It is known that, when Alice uses any deterministic algorithm/protocol, there
always exist messages~$m_{\mathrm{Alice}}$ and $m_{\mathrm{Bob}}$ such that
the number of communication bits is at least $\ell$; see
Kushilevitz~\cite{Kushilevitz:1997:communication-complexity}.

For the proof below, we will use the assumptions made in
Section~\ref{sec:setup}, but again we do not require multiplicative independence.

\begin{proof}[Proof of Theorem~\ref{thm:main}, part~\ref{itm:lower}]
  We assume for contradiction that for each~$n$, there exists a
  multi-base representation with only $\oh[\big]{\frac{\log n}{\log \log n}}$
  summands. Set $\ell = \floor{\log n}$.
  Let $m_{\mathrm{Alice}}$ and $m_{\mathrm{Bob}}$ be $\ell$-bit
  messages so that~$\ell$ bits need to be communicated in order
  to determine equality.

  Now, suppose that Alice converts the $\ell$-bit message
  $m_{\mathrm{Alice}}$ to a multi-base representation with
  $\oh[\big]{\frac{\log n}{\log \log n}}$ summands.
  Since all exponents are in $\Oh{\log n}$ and the number~$r$ of bases is fixed,
  each summand of a multi-base representation can be denoted by $\Oh{\log \log n}$ bits.
  Therefore, Alice can tell Bob the whole message~$m_A$ by only
  \begin{equation*}
    \Oh{\log \log n} \cdot \oh[\Big]{\frac{\log n}{\log \log n}}
    = \oh{\log n}
  \end{equation*}
  bits; a contradiction.
\end{proof}

\renewcommand{\MR}[1]{}
\bibliography{bib/cheub}
\bibliographystyle{amsplain}

\end{document}